\documentclass[12pt,a4paper]{article}
\usepackage{amsmath,amssymb,amsthm,bm}
\usepackage{graphicx}
\usepackage{txfonts}

\newtheorem{theorem}{Theorem}[section]
\newtheorem{proposition}[theorem]{Proposition}
\newtheorem{corollary}[theorem]{Corollary}
\newtheorem{lemma}[theorem]{Lemma}
\newtheorem{remark}[theorem]{Remark}
\newtheorem{example}[theorem]{Example}

\newcommand{\1}{{\bm 1}}

\numberwithin{equation}{section}

\begin{document}

\begin{center}
\large\bf
On Quadratic Embedding Constants of Star Product Graphs
\end{center}

\bigskip

\begin{center}
Wojciech M{\l}otkowski \\
Instytut Matematyczny \\
Uniwersytet Wroc{\l}awski \\
Plac~Grunwaldzki~2/4, 50-384 Wroc{\l}aw, Poland \\
mlotkow@math.uni.wroc.pl
\\
and
\\
Nobuaki Obata\\
Graduate School of Information Sciences \\
Tohoku University\\
Sendai 980-8579 Japan \\
obata@math.is.tohoku.ac.jp
\end{center}

\bigskip

\begin{quote}
\textbf{Abstract}\enspace
A connected graph $G$ is of QE class if it admits a
quadratic embedding in a Hilbert space,
or equivalently if the distance matrix is conditionally
negative definite,
or equivalently if the quadratic embedding constant 
$\mathrm{QEC}(G)$ is non-positive.
For a finite star product of (finite or infinite) graphs
$G=G_1\star\dotsb \star G_r$ 
an estimate of $\mathrm{QEC}(G)$
is obtained after a detailed analysis of the minimal solution of
a certain algebraic equation.
For the path graph $P_n$ 
an implicit formula for $\mathrm{QEC}(P_n)$ is derived,
and by limit argument  
$\mathrm{QEC}(\mathbb{Z})=\mathrm{QEC}(\mathbb{Z}_+)=-1/2$ is shown.
During the discussion a new integer sequence is found.
\end{quote}

\begin{quote}
\end{quote}

\begin{quote}
\end{quote}

%%%%%%%%%%%%%%%%%%%%%%
\section{Introduction}

Let $G=(V,E)$ be a (finite or infinite) connected graph.
A map $\varphi$ from $V$ into a Hilbert space $\mathcal{H}$
(of finite or infinite dimension) is called a
\textit{quadratic embedding} if it fulfills
\begin{equation}\label{01eqn:quadratic embedding}
\|\varphi(x)-\varphi(y)\|^2=d(x,y),
\qquad
x,y\in V,
\end{equation}
where $\|\cdot\|$ stands for the norm of $\mathcal{H}$,
and $d(x,y)$ the graph distance between two vertices $x,y\in V$,
i.e., the length of a shortest walk (or path) connecting $x$ and $y$.
A graph $G$ is said to be \textit{of QE class}
if it admits a quadratic embedding.
Graphs of QE class have been studied along with
graph theory \cite{Balaji-Bapat2007}, \cite{Bapat2010}, \cite{Mlotkowski1992},
Euclidean distance geometry
\cite{Jaklic-Modic2010}, \cite{Jaklic-Modic2013},
\cite{Jaklic-Modic2014}, \cite{Liberti-Lavor-Maculan-Mucherino2014},
and so forth.
They have appeared also in quantum probability
and non-commutative harmonic analysis
\cite{Bozejko88}, \cite{Bozejko89}, \cite{Haagerup79},
\cite{Hora-Obata2007}, \cite{Obata2007}, \cite{Obata2011}.

It follows from
the result of Schoenberg \cite{Schoenberg1935},
\cite{Schoenberg1938}
(also Young--Householder \cite{Young-Householder1938}).
that $G$ is of QE class
if and only if the distance matrix $D = [d(x,y)]$ is
conditionally negative definite.
It is then natural to consider, as a quantitative approach,
the \textit{QE constant} of a graph $G$ defined by
\begin{equation}\label{01eqn:QEC}
\mathrm{QEC}(G)=\sup\{\langle f,Df\rangle\,;\,
f\in C_0(V), \, \langle f,f\rangle=1, \langle \1, f\rangle=0\},
\end{equation}
where $C_0(V)$ is the space of $\mathbb{R}$-valued functions on $V$
with finite supports,
and $\langle\cdot,\cdot\rangle$ is the canonical inner product
on $C_0(V)$.
Moreover, $\langle \1,f\rangle=\sum_{x\in V}f(x)$
by overuse of symbols, where $\1(x)=1$ for all $x\in V$.
Obviously, $G$ is of QE class if and only if $\mathrm{QEC}(G)\le0$.
The QE constant has been introduced in the recent paper
\cite{Obata-Zakiyyah2017},
where graph operations preserving the property of
QE class are discussed and the QE constants
of graphs on $n\le5$ vertices are listed.
Moreover, for a particular class of graphs 
distance spectrum (for generalities see \cite{Aouchiche-Hansen2014})
is useful for calculating the QE constants,
but the relation is not clear in general.

In this paper, we focus on the star product
as one of the most elementary graph operations.
Given graphs $G_j = (V_j, E_j)$ with distinguished vertices
$o_j\in V_j$, $1\le j\le r$,
the star product
\[
G_1\star\dotsb\star G_r
=(G_1,o_1)\star\dotsb\star (G_r,o_r)
\]
is by definition a graph obtained by glueing
graphs $G_j$ at the vertices $o_j$.
It is known (see e.g., \cite{Obata-Zakiyyah2017} for
an explicit statement) that a star product of two graphs
of QE class is again of QE class.
An equivalent property appears 
in the study of length functions on Cayley graphs,
of which the root traces back to Haagerup \cite{Haagerup79},
see also Bo\.zejko--Januszkiewicz--Spatzier \cite{Bozejko88}
and Bo\.zejko \cite{Bozejko89}.
However, a concise formula for the QE constant
of a star product is not known.
Our goal of this paper is
to derive an implicit description of 
$\mathrm{QEC}(G_1\star \dotsb\star G_r)$ and obtain
a sufficiently good estimate of it
in terms of $Q_j=\mathrm{QEC}(G_j)$.
The main results are stated in Theorems
\ref{01thm:QEC=0}, \ref{04thm:main estimate}
and their corollaries.

This paper is organized as follows.
In Section \ref{Sec:Preliminaries} we derive some estimates
of the minimal solution of an algebraic equation of the 
following type:
\begin{equation}\label{0eqn:abstractimportantequation}
\sum_{j=1}^{r}\frac{d_j}{a_jd_j +a_j-\lambda}=\frac{1}{\lambda}\,.
\end{equation}
In Section \ref{sec:Conditional Minimum} we study the conditional
minimum of a quadratic function of the following type:
\begin{equation}\label{0eqn:main quadratic form}
\phi(x_0,\bm{x}_1,\dots, \bm{x}_r)
=\sum_{j=1}^{r} a_j\left(\langle\bm{x}_j,\bm{x}_j\rangle
 +\langle \mathbf{1}_j,\bm{x}_j\rangle^2\right)
\end{equation}
subject to conditions:
\begin{equation}\label{0eqn:condition (1)}
x_0^2+\sum_{j=1}^{r}\langle\bm{x}_j,\bm{x}_j\rangle=1,
\qquad
x_0+\sum_{j=1}^{r}\langle \bm{1}_j,\bm{x}_j\rangle=0.
\end{equation}
We show that the conditional minimum of
\eqref{0eqn:main quadratic form} coincides with
the minimal solution of
\eqref{0eqn:abstractimportantequation}.
With these results we prove the
main theorem in Section \ref{Sec:Star product graphs}
and mention some relevant results and problems.
In Section \ref{Sec:Infinite graphs} we discuss infinite graphs,
in particular, infinite path graphs
$\mathbb{Z}_+$ and $\mathbb{Z}$.
The QE constant of a finite path $P_n$ for a general $n$ is
not known explicitly. 
We derive an indirect formula for $\mathrm{QEC}(P_n)$
and by taking limit we obtain
$\mathrm{QEC}(\mathbb{Z}_+)=\mathrm{QEC}(\mathbb{Z})=-1/2$.
Finally, in Section \ref{Sec:Appendix} 
we study some combinatorial identities used in the estimate of
$\mathrm{QEC}(P_n)$ and find a new integer sequence
which is interesting for itself.

\bigskip
{\bfseries Acknowledgements:}
WM is supported by NCN grant 2016/21/B/ST1/00628.
NO is supported by JSPS Grant-in-Aid for Scientific Research
No.~16H03939.
He thanks the Institute of Mathematics, University of Wroc{\l}aw
for their kind hospitality.

%%%%%%%%%%%%%%%%%%%%%%%%%%%%%%%%%%%%%%%%%
\section{Preliminaries}
\label{Sec:Preliminaries}

Given a natural number $r\ge1$ and a pair of parameter vectors
\[
\bm{a}=(a_1,a_2,\dots, a_r),
\qquad
\bm{d}=(d_1,d_2,\dots, d_r),
\]
we consider an algebraic equation of the following type:
\begin{equation}\label{01eqn:abstractimportantequation}
\sum_{j=1}^{r}\frac{d_j}{a_jd_j +a_j-\lambda}=\frac{1}{\lambda}\,.
\end{equation}
The parameters $\bm{a}$ and $\bm{d}$ are always assumed
to fulfill the following conditions:
\begin{enumerate}
\setlength{\itemsep}{0pt}
\item[(i)] $a_1,\dots,a_r$ are positive real numbers,
\item[(ii)] $d_1,\dots,d_r$ are positive real numbers or
$\infty$. If $d_j=\infty$, we understand that
\[
\frac{d_j}{a_jd_j +a_j-\lambda}
=\frac{\infty}{a_j\cdot\infty+a_j-\lambda}=\frac{1}{a_j}.
\]
\end{enumerate}

\subsection{Separation of solutions}

Given $\bm{a}=(a_1,a_2,\dots, a_r)$ and
$\bm{d}=(d_1,d_2,\dots, d_r)$,
arranging $a_jd_j+a_j$ in order,
we write
\[
\{a_1d_1+a_1,\dots, a_rd_r+a_r\}=\{c_1<\dots<c_s\}
\]
and set $c_0=0$.
It may happen that $c_s=\infty$.

\begin{proposition}\label{01prop:01}
Every open interval $(c_{i-1},c_i)$, $1\le i\le s$,
contains exactly one solution $\lambda_i$ of
\eqref{01eqn:abstractimportantequation}.
Moreover, these $\lambda_1,\dots,\lambda_s$ are all the
solutions of \eqref{01eqn:abstractimportantequation}.
\end{proposition}

\begin{proof}
We set
\begin{equation}\label{01eqn:functionf}
f(\lambda)=\sum_{j=1}^r \frac{d_j}{a_jd_j+a_j-\lambda}-\frac{1}{\lambda}\,,
\end{equation}
which becomes
\begin{equation}\label{01eqn:rewritten}
f(\lambda)=\sum_{i=1}^s \frac{d_i^\prime}{c_i-\lambda}-\frac{1}{\lambda}
\end{equation}
with some $d_i^\prime>0$.
If $c_s=\infty$, then $d_{s}'/(c_s-\lambda)$ becomes 
a positive constant.
Hence, for any $1\le i \le s$, the function $f(\lambda)$ is
strictly increasing on the interval $(c_{i-1},c_i)$
as a sum of increasing functions. 
Moreover,
for any $1\le i\le s$ with $c_i<\infty$ we have
\[
\lim_{\lambda\to c_{i-1}+0}f(\lambda)=-\infty,
\qquad
\lim_{\lambda\to c_{i}-0}f(\lambda)=+\infty.
\]
If $c_s=\infty$,
we have $\lim_{\lambda\to\infty}f(\lambda)>0$.
While, if $c_s<\infty$, we have $f(\lambda)<0$ for all $\lambda>c_s$.
Hence every interval $(c_{i-1},c_i)$, $1\le i\le s$, contains exactly
one solution $\lambda_i$ of \eqref{01eqn:abstractimportantequation}.
Since the equation \eqref{01eqn:abstractimportantequation}
is equivalent to an algebraic equation of degree $s$
as is seen from \eqref{01eqn:rewritten},
$\{\lambda_1,\dots,\lambda_s\}$ exhaust its solutions.
\end{proof}

%%%%%%%%%%%%%%%%%%%%%%%%%%%%%%%%%%%%%%%%
\subsection{Estimate of the minimal solution}

Let $\lambda_1(\bm{d},\bm{a})$ denote the minimal solution
of \eqref{01eqn:abstractimportantequation},
which verifies $\lambda_1(\bm{d},\bm{a})>0$ by Proposition~\ref{01prop:01}.
In fact, for $r=1$ we have 
\begin{equation}\label{02eqn:lambda for r=1}
\lambda_1(\bm{d},\bm{a})=a_1
\end{equation}
and for $r=2$,
\begin{align}
\lambda_1(\bm{d},\bm{a})
&=\frac{2a_1a_2}
  {a_1+a_2+\sqrt{(a_1+a_2)^2-
   \dfrac{4(d_1+d_2+1)}{(d_1+1)(d_2+1)}\,a_1a_2}}
\nonumber \\
&=\frac{2a_1a_2}
  {a_1+a_2+\sqrt{(a_1-a_2)^2+
   \dfrac{4d_1d_2}{(d_1+1)(d_2+1)}\,a_1a_2}}\,.
\label{02eqn:lambda for r=2}
\end{align}
It is difficult to obtain a concise description of
$\lambda_1(\bm{d},\bm{a})$ for $r\ge3$ in general.
Instead, we will obtain good estimates for
$\lambda_1(\bm{d},\bm{a})$ useful in applications.

\begin{proposition}\label{02prop:020}
Let $r\ge2$.
The minimal solution $\lambda_1(\bm{d},\bm{a})$ of
\eqref{01eqn:abstractimportantequation} satisfies
\begin{equation}\label{02eqn:simple estimate of lambda_1}
\left(\dfrac{1}{a_1}+\dots+\dfrac{1}{a_r}\right)^{-1}
\le\lambda_1(\bm{d},\bm{a})
<\min\{a_1,\dots,a_r\},
\end{equation}
where the equality holds if and only if $d_1=\dots=d_r=\infty$.
\end{proposition}

\begin{proof}
%Note first that every solution of
%\eqref{01eqn:abstractimportantequation} is positive,
%as an immediate consequence of Proposition \ref{01prop:01}.
We first show the right-half of
\eqref{02eqn:simple estimate of lambda_1}.
Let $a_{j_0}=\min\{a_1,\dots,a_r\}$.
Since $a_{j_0}\le a_j<a_jd_j+a_j$ for all $j$, we have $0=c_0<a_{j_0}<c_1$.
Moreover, letting $f(\lambda)$ be as in \eqref{01eqn:functionf},
we have
\begin{align*}
f(a_{j_0})
&=\frac{d_{j_0}}{a_{j_0} d_{j_0}+a_{j_0}-a_{j_0}}
+\sum_{\substack{j=1,\\j\ne j_0}}^{r}\frac{d_j}{a_j d_j+a_j-a_{j_0}}-\frac{1}{a_{j_0}} \\
&=\sum_{\substack{j=1,\\j\ne j_0}}^{r}\frac{d_j}{a_j d_j+a_j-a_{j_0}}
>0.
\end{align*}
Since $f(\lambda)$ is increasing on the interval
$(0,c_1)$, we see that $\lambda_1(\bm{d},\bm{a})<a_{j_0}$.

Now we are going to prove the left-half
of \eqref{02eqn:simple estimate of lambda_1}.
%Assuming that $a_1=\min\{a_1,\dots,a_r\}$, we keep the same notation
%as in the proof of Proposition \ref{01prop:01}.
For simplicity we set
\[
\lambda_0=\left(\dfrac{1}{a_1}+\dots+\dfrac{1}{a_r}\right)^{-1}.
\]
Obviously, for $1\le j\le r$ we have $0<\lambda_0<a_j$ and hence
\[
\frac{d_j}{a_jd_j+a_j-\lambda_0}\le\frac{1}{a_j},
\]
where the equality holds if and only if $d_j=\infty$.
Taking the sum over $1\le j\le r$ we get
\[
\sum_{j=1}^r \frac{d_j}{a_jd_j+a_j-\lambda_0}
\le\sum_{j=1}^r \frac{1}{a_j}
=\frac{1}{\lambda_0}\,,
\]
from which we see that $f(\lambda_0)\le0$
and the equality holds if and only if
$d_1=\dots=d_r=\infty$.
Since $0<\lambda_0<a_{j_0}<c_1$ and $f(\lambda)$ is increasing on the
interval $(0,c_1)$, we have $\lambda_0\le\lambda_1(\bm{d},\bm{a})$,
which shows the left-half
of \eqref{02eqn:simple estimate of lambda_1}.
\end{proof}

\subsection{Sharper estimates}

We will sharpen the estimate \eqref{02eqn:simple estimate of lambda_1}.
For $\bm{a}=(a_1,\dots,a_r)$
and $\bm{a}^\prime=(a_1^\prime,\dots,a_r^\prime)$
we write
$\bm{a}\le \bm{a}^\prime$ if
$a_j\le a_j^\prime$ for all $1\le j\le r$.
Similarly, we define $\bm{d}\le \bm{d}^\prime$.
The following comparison is useful.

\begin{proposition}\label{04prop:comparison}
If $\bm{d}\ge\bm{d}^\prime$ and $\bm{a}\le\bm{a}^\prime$,
then $\lambda_1(\bm{d},\bm{a})\le
\lambda_1(\bm{d}^\prime,\bm{a}^\prime)$.
Moreover, if $\bm{d}\neq\bm{d}^\prime$ or $\bm{a}\neq\bm{a}^\prime$
in addition, we have
$\lambda_1(\bm{d},\bm{a})<\lambda_1(\bm{d}^\prime,\bm{a}^\prime)$.
\end{proposition}

\begin{proof}
Suppose that $0<d_j^\prime \le d_j\le\infty$
and $0<a_j\le a_j^\prime$.
Then, by elementary algebra we obtain
\begin{equation}\label{04eqn:in proof 4.8}
\frac{d_j}{a_jd_j+a_j-\lambda}
\ge\frac{d^\prime_j}{a_j^\prime d_j^\prime+a^\prime_j-\lambda},
\qquad
0<\lambda< a_j.
\end{equation}
Moreover, the strict inequality holds
if $0<d_j^\prime<d_j\le\infty$ or $0<a_j< a_j^\prime$.
Put
\[
f(\lambda)
=\sum_{j=1}^{r}\frac{d_j}{a_jd_j +a_j-\lambda}-\frac{1}{\lambda},
\qquad
g(\lambda)
=\sum_{j=1}^{r}
 \frac{d^\prime_j}{a_j^\prime d_j^\prime +a_j^\prime-\lambda}
 -\frac{1}{\lambda}.
\]
Now suppose that
$\bm{d}\ge\bm{d}^\prime$ and $\bm{a}\le\bm{a}^\prime$.
It then follows from \eqref{04eqn:in proof 4.8} that
$f(\lambda)\ge g(\lambda)$ for
$0<\lambda< \min\{a_1,\dots,a_r\}$,
and hence for $0<\lambda\le\lambda_1(\bm{d},\bm{a})$.
Therefore, $\lambda_1(\bm{d},\bm{a})\le
\lambda_1(\bm{d}^\prime,\bm{a}^\prime)$.
If $\bm{d}\neq\bm{d}^\prime$ or $\bm{a}\neq\bm{a}^\prime$,
we have $f(\lambda)> g(\lambda)$ for
$0<\lambda\le\lambda_1(\bm{d},\bm{a})$, which yields
$\lambda_1(\bm{d},\bm{a})<\lambda_1(\bm{d}^\prime,\bm{a}^\prime)$.
\end{proof}

As an immediate consequence of
Proposition \ref{04prop:comparison}, we have
\begin{equation}\label{estimationfrombelow0}
\left(\frac{1}{a_1}+\dots+\frac{1}{a_r}\right)^{-1}
=\lambda_{1}(\bm{\infty},\bm{a})
<\lambda_{1}(\bm{d},\bm{a})
\end{equation}
for any $\bm{d}\neq \bm{\infty}=(\infty,\dots,\infty)$.
Note that \eqref{estimationfrombelow0} is reproduction
of Proposition \ref{02prop:020}.

\begin{proposition}\label{gotoinfinitypropeq}
We have
\[
\lambda_1(\bm{d},\bm{a})
=\inf\left\{\lambda_1(\bm{e},\bm{a})\,;\,
\begin{array}{l}
\bm{e}=(e_1,\dots,e_r)\le \bm{d}, \\
e_1<\infty,\dots, e_r<\infty
\end{array}
\right\},
\]
or equivalently,
\[
\lambda_1(\bm{d},\bm{a})
=\lim_{n\to\infty}\lambda_1(\bm{d}\wedge\bm{n},\bm{a}),
\]
where $\bm{d}\wedge\bm{n}=(d_1\wedge n,\dots,d_r\wedge n)$.
\end{proposition}

\begin{proof}
Here $\bm{a}=(a_1,\dots,a_r)$ is fixed.
Substituting $d_j\mapsto 1/u_j$ we define
\[
F(u_1,\dots,u_r,\lambda)
=\sum_{j=1}^{r}\frac{1}{a_j u_j+a_j-\lambda u_j}-\frac{1}{\lambda}\,.
\]
Then the equation $F(u_1,\dots,u_r,\lambda)=0$ gives rise to
an implicit function $\lambda=g(u_1,\dots,u_r)$
with the initial condition
$g(0,\dots,0)=\lambda_0=(1/a_1+\dots+1/a_r)^{-1}$.
It suffices to show that $g$ is well-defined
and is continuous on $[0,\infty)^{r}$.
We know that the minimal solution
\[
\lambda=g(u_1,\dots,u_r)
=\lambda_1\bigg(\frac{1}{u_1},\dots,\frac{1}{u_r},\bm{a}\bigg)
\]
exists for all $\bm{u}\in[0,\infty)^{r}$.
On the other hand for such $\bm{u}$ and $\lambda$ we have
\[
\frac{\partial F}{\partial\lambda}
=\sum_{j=1}^{r}\frac{u_j}
{(a_j u_j+a_j-\lambda u_j)^2}+\frac{1}{\lambda^2}>0,
\]
which implies that $g$ is continuous on $[0,\infty)^{r}$.
\end{proof}

Hereafter in this subsection we assume that $\bm{d}\neq\bm{\infty}$,
namely, $\bm{d}=(d_1,\dots,d_r)$ with $d_j<\infty$ for some $j$.
As before, we put
\[
c_1=\min\{a_1d_1+a_1,\dots, a_rd_r+a_r\}.
\]

\begin{proposition}
We have
\begin{equation}\label{estimation1}
\left(\frac{1}{c_1}+\sum_{j=1}^{r}\frac{d_j}{d_j a_j+a_j}\right)^{-1}
\le\lambda_1(\bm{d},\bm{a})
<\left(\sum_{j=1}^{r}\frac{d_j}{d_j a_j+a_j}\right)^{-1}
\end{equation}
and the equality holds if and only if $d_1 a_1+a_1=\dots=d_r a_r+a_r$.
\end{proposition}

\begin{proof}
Let $\lambda'$ and $\lambda''$ denote
the left- and right-hand sides of \eqref{estimation1}, respectively.
First we note that for $0<\lambda<c_1$ we have
\[
\frac{d_j}{d_j a_j+a_j-\lambda}\le \frac{d_j c_1}{(c_1-\lambda)(d_j+1)a_j},
\]
where the equality holds if and only if $c_1=d_j a_j+a_j$.
Therefore the solution of \eqref{01eqn:abstractimportantequation}
in the interval $(0,c_1)$ is greater than the solution of
\[
\sum_{j=1}^{r}\frac{d_j c_1}{(c_1-\lambda)(d_j+1)a_j}=\frac{1}{\lambda},
\]
which is exactly $\lambda'$.
For the second inequality in \eqref{estimation1}
we can assume that $\lambda''<c_1$
(for otherwise $\lambda_1(\bm{d},\bm{a})<\min\{a_1,\dots,a_r\}<c_1\le \lambda''$).
Then we have
\[
\frac{d_j}{d_j a_j+a_j-\lambda''}\ge \frac{d_j}{d_j a_j+a_j},
\]
with equality only when $d_j=\infty$.
Taking the sum over $j=1,\dots,r$ we get
\[
\sum_{j=1}^{r}\frac{d_j}{d_j a_j+a_j-\lambda''}>\frac{1}{\lambda''},
\]
which implies $\lambda_1(\bm{d},\bm{a})<\lambda''$
\end{proof}

Here is slightly more precise estimation from below.

\begin{proposition}
We have
\begin{equation}\label{estimation2}
c_1 \left(1+\sum_{j=1}^{r}\frac{d_j(c_1-\lambda_0)}{d_j a_j+a_j-\lambda_0}\right)^{-1}
\le\lambda_1(\bm{d},\bm{a}),
\end{equation}
where
\[
\lambda_0=\bigg(\frac{1}{a_1}+\dots+\frac{1}{a_r}\bigg)^{-1}
\]
and the equality holds if and only if $d_1 a_1+a_1=\dots=d_r a_r+a_r$.
\end{proposition}

\begin{proof}
For $1\le j\le r$ and $\lambda_0<\lambda<c_1$ we have
\[
\frac{d_j}{d_j a_j+a_j-\lambda}
\le\frac{d_j(c_1-\lambda_0)}{(d_j a_j+a_j-\lambda_0)(c_1-\lambda)},
\]
where the equality holds if and only if $c_1=d_j a_j+a_j$.
Therefore the minimal solution of \eqref{01eqn:abstractimportantequation} is greater
(or equal if $d_1 a_1+a_1=\dots=d_r a_r+a_r$) than the solution of
\[
\sum_{j=1}^{r}\frac{d_j(c_1-\lambda_0)}{(d_j a_j+a_j-\lambda_0)(c_1-\lambda)}=\frac{1}{\lambda},
\]
which is the left hand side of \eqref{estimation2}.
\end{proof}

One can check that \eqref{estimation2} gives
a more precise estimate of $\lambda_1(\bm{d},\bm{a})$
from below than \eqref{estimation1},
which is still better than \eqref{estimationfrombelow0}, i.e.,
\begin{align*}
\left(\sum_{j=1}^{r}\frac{1}{a_j}\right)^{-1}
&<\left(\frac{1}{c_1}
  +\sum_{j=1}^{r}\frac{d_j}{d_j a_j+a_j}\right)^{-1} \\
&\le c_1 \left(1+\sum_{j=1}^{r}\frac{d_j(c_1-\lambda_0)}{d_j a_j+a_j-\lambda_0}\right)^{-1}
\le\lambda_1(\bm{d},\bm{a}),
\end{align*}
with equalities if and only if $d_1 a_1+a_1=\dots=d_r a_r+a_r$.

%%%%%%%%%%%%%%%%%%%%%%%%%%%%%%%%%%%%%%%%%%%%%%%%%%%%%%%%%%%%%%%%%
\section{Conditional Minimum of a Quadratic Function}
\label{sec:Conditional Minimum}

Given a natural number $r\ge1$, and a pair of parameter vectors
\[
\bm{a}=(a_1,a_2,\dots, a_r),
\qquad
\bm{d}=(d_1,d_2,\dots, d_r),
\]
satisfying conditions:
\begin{enumerate}
\setlength{\itemsep}{0pt}
\item[(i)] $a_1,\dots,a_r\ge0$ are non-negative real numbers,
\item[(ii)] $d_1,\dots,d_r\ge1$ are natural numbers or $\infty$,
\end{enumerate}
we consider a quadratic function
in $1+d_1+\dots+d_r$ variables of the following form:
\begin{equation}\label{02eqn:main quadratic form}
\phi(x_0,\bm{x})
=\phi(x_0,\bm{x}_1,\dots, \bm{x}_r)
=\sum_{j=1}^{r} a_j\left(\langle\bm{x}_j,\bm{x}_j\rangle
 +\langle \mathbf{1}_j,\bm{x}_j\rangle^2\right),
\end{equation}
where $x_0\in\mathbb{R}$, $\bm{x}_j\in\mathbb{R}^{d_j}$,
$\mathbf{1}_j=[1\,\,\dots \,\,1]^{\mathrm{T}}\in\mathbb{R}^{d_j}$,
and $\langle\cdot,\cdot\rangle$ stands for the
canonical inner product.
In case of $d_j=\infty$ we always assume that vectors
$\bm{x}_j\in\mathbb{R}^{d_j}$ have finite supports,
that is, the entries of $\bm{x}_j$ vanish except finitely many ones.
For such vectors
$\langle\bm{x}_j,\bm{x}_j\rangle$ and
$\langle \mathbf{1}_j,\bm{x}_j\rangle$ are defined as finite sums
and $\phi(x_0,\bm{x})$
is defined on the set of vectors with finite supports.
Note also that the right-hand side of \eqref{02eqn:main quadratic form}
is free from the variable $x_0$.

Let $M(\bm{d},\bm{a})$ denote the conditional infimum:
\[
M(\bm{d},\bm{a})
=\inf \phi(x_0,\bm{x}),
\]
where the infimum is taken over the vectors
$(x_0,\bm{x})$ with finite supports,
fulfilling the conditions:
\begin{align}
&x_0^2+\sum_{j=1}^{r}\langle\bm{x}_j,\bm{x}_j\rangle=1,
\label{03eqn:condition (1)}\\
&x_0+\sum_{j=1}^{r}\langle \bm{1}_j,\bm{x}_j\rangle=0.
\label{03eqn:condition (2)}
\end{align}
If $d_j<\infty$ for all $1\le j\le r$,
we prefer to call $M(\bm{d},\bm{a})$ the conditional mimimum
rather than infimum.

Although $M(\bm{d},\bm{a})$ itself is defined for any choice of
real numbers $\bm{a}=(a_1,\dots,a_r)$,
the condition (i) above is posed for our application.

\subsection{Elementary properties of $M(\bm{d},\bm{a})$}

\begin{proposition}\label{03prop:M(d,a)le min}
We have $0\le M(\bm{d},\bm{a})\le \min\{a_1,\dots,a_r\}$.
\end{proposition}

\begin{proof}
It is obvious from definition that $\phi(x_0,\bm{x})\ge0$
for all $x_0$ and $\bm{x}$, so that $M(\bm{d},\bm{a})\ge0$.
Setting $\bm{x}_2=\dots=\bm{x}_r=\bm{0}$ and
taking $x_0$ and $\bm{x}_1$ in such a way that
\[
x_0^2+\langle\bm{x}_1,\bm{x}_1\rangle=1,
\qquad
x_0+\langle \bm{1}_1,\bm{x}_1\rangle=0,
\]
we see that $\phi(x_0,\bm{x})$ becomes
\[
a_1\left(\langle\bm{x}_1,\bm{x}_1\rangle
 +\langle \mathbf{1}_1,\bm{x}_1\rangle^2\right)
=a_1\left((1-x_0^2)+(-x_0)^2\right)
=a_1\,.
\]
Hence the function $\phi(x_0,\bm{x})$ attains the value $a_1$
under conditions \eqref{03eqn:condition (1)}
and \eqref{03eqn:condition (2)}.
Similarly, it attains the value $a_j$ for $1\le j\le r$.
Therefore, the conditional infimum verifies $M(\bm{d},\bm{a})
\le \min\{a_1,\dots,a_r\}$.
\end{proof}

\begin{proposition}
If $\bm{a}\le\bm{b}$ and $\bm{d}\le \bm{e}$, we have
\[
M(\bm{e},\bm{a})\le M(\bm{d},\bm{a})\le M(\bm{d},\bm{b}).
\]
\end{proposition}

\begin{proof}
Straightforward by definition.
\end{proof}

\begin{proposition}\label{03prop:trivial case}
If $a_j=0$ for some $1\le j\le r$,
we have $M(\bm{d},\bm{a})=0$.
\end{proposition}

\begin{proof}
Immediate from Proposition \ref{03prop:M(d,a)le min}.
\end{proof}

\begin{proposition}\label{03prop:trivial case(2)}
For $r=1$ we have $M(\bm{d},\bm{a})=a_1$.
\end{proposition}

\begin{proof}
In fact, $\phi(x_0,\bm{x})$ is constant under
\eqref{03eqn:condition (1)} and \eqref{03eqn:condition (2)} as
\[
\phi(x_0,\bm{x})
=a_1\left(\langle\bm{x}_1,\bm{x}_1\rangle
 +\langle \mathbf{1}_1,\bm{x}_1\rangle^2\right)
=a_1\left((1-x_0^2)+(-x_0)^2\right)
=a_1.
\]
Therefore, $M(\bm{d},\bm{a})=a_1$.
\end{proof}

%%%%%%%%%%%%%%%%%%%%%%%%%%%%%%%%%%%%%%%%%%%%%%%%%%%%%%
\subsection{A Characterization of $M(\bm{d},\bm{a})$}

\begin{theorem}\label{02thm:Characterization of M(d,a)}
Let $r\ge1$.
Assume that $a_j>0$ and $1\le d_j\le\infty$ for all $1\le j\le r$.
Then $M(\bm{d},\bm{a})$ coincides with
the minimal solution of
\begin{equation}\label{02eqn:key equation 100}
\sum_{j=1}^{r}\frac{d_j}{a_j d_j+a_j-\lambda}=\frac{1}{\lambda}\,.
\end{equation}
In other words, with the notations introduced in Section 
\ref{Sec:Preliminaries}, we have
\begin{equation}\label{03eqn:M=lambda_1}
M(\bm{d},\bm{a})=\lambda_1(\bm{d},\bm{a}).
\end{equation}
\end{theorem}

For $r=1$ the assertion
in Theorem \ref{02thm:Characterization of M(d,a)} is immediate.
In fact, the unique solution of
\[
\frac{d_1}{a_1 d_1+a_1-\lambda}=\frac{1}{\lambda}
\]
is $\lambda=a_1$.
On the other hand, we have $M(\bm{d},\bm{a})=a_1$ by
Proposition \ref{03prop:trivial case(2)}.

In the rest of this subsection,
we will prove Theorem \ref{02thm:Characterization of M(d,a)}
under the condition that $r\ge2$,
$a_j>0$ and $1\le d_j<\infty$ for all $1\le j\le r$.
The limit case will be treated in the next subsection.

Employing the method of Lagrange multipliers, we set
\[
F(x_0,\bm{x},\lambda,\mu)
=\phi(x_0,\bm{x})-\lambda(g(x_0,\bm{x})-1)-\mu h(x_0,\bm{x}),
\]
where
\begin{align}
g(x_0,\bm{x})
&=x_0^2+\sum_{j=1}^{r}\langle\bm{x}_j,\bm{x}_j\rangle,
\label{02eqn:def of g}\\
h(x_0,\bm{x})
&=x_0+\sum_{j=1}^{r}\langle \mathbf{1}_j,\bm{x}_j\rangle.
\label{02eqn:def of h}
\end{align}
Let $\mathcal{S}$ be the set of stationary points of
$F(x_0,\bm{x},\lambda,\mu)$, namely, the set of solutions of
the system of equations:
\begin{align}
&\frac{\partial F}{\partial x_0}=0,
\label{02eqn:stationary points of F(0)}\\
&\frac{\partial F}{\partial \bm{x}_j}=\bm{0},
\quad 1\le j\le r,
\label{02eqn:stationary points of F(j)}\\
&\frac{\partial F}{\partial \lambda}
=\frac{\partial F}{\partial \mu}
=0,
\label{02eqn:stationary points of F(condition)}
\end{align}
where
\[
\frac{\partial}{\partial \bm{x}_j}
=\left[
\frac{\partial}{\partial x_{j1}} \,\,
\dots \,\,
\frac{\partial}{\partial x_{jd_j}}
\right]^{\mathrm{T}},
\qquad
\bm{x}_j=\Big[x_{j1} \,\, \dots \,\, x_{jd_j}\Big]^{\mathrm{T}}.
\]
Since conditions \eqref{03eqn:condition (1)}
and \eqref{03eqn:condition (2)} determine a smooth compact
manifold (in fact, a sphere of dimension $d_1+\dots+d_r-1\ge1$),
the conditional minimum of $\phi(x_0,\bm{x})$ is found from
the stationary points of $F(x_0,\bm{x},\lambda,\mu)$
in such a way that
\begin{equation}\label{02eqn:M(b,a)00}
M(\bm{d},\bm{a})=\min\{\phi(x_0,\bm{x})\,;\,
(x_0,\bm{x},\lambda,\mu)\in\mathcal{S}\}.
\end{equation}

We will first obtain explicit forms of
\eqref{02eqn:stationary points of F(0)}
and \eqref{02eqn:stationary points of F(j)}.
Applying elementary calculus to \eqref{02eqn:main quadratic form},
we come to
\[
\frac{\partial}{\partial x_0}\phi(x_0,\bm{x})=0,
\qquad
\frac{\partial}{\partial \bm{x}_j}\phi(x_0,\bm{x})
=2a_{j}\bm{x}_j+2a_j\langle\bm{1}_j,\bm{x}_j\rangle\bm{1}_j\,.
\]
Similarly, from \eqref{02eqn:def of g} and \eqref{02eqn:def of h}
we obtain
\begin{gather*}
\frac{\partial}{\partial x_0}g(x_0,\bm{x})=2x_0,
\qquad
\frac{\partial}{\partial \bm{x}_j}g(x_0,\bm{x})=2\bm{x}_j, \\
\frac{\partial}{\partial x_0}h(x_0,\bm{x})=1,
\qquad
\frac{\partial}{\partial \bm{x}_j}h(x_0,\bm{x})=\bm{1}_j\,.
\end{gather*}
Thus, \eqref{02eqn:stationary points of F(0)}
and \eqref{02eqn:stationary points of F(j)} are 
respectively equivalent to
\begin{equation}\label{vectorequation0}
2\lambda x_0+\mu=0,
\end{equation}
and
\begin{equation}\label{03eqn:vectorequation}
2a_j\bm{x}_j+2a_j\langle \bm{1}_j,\bm{x}_j\rangle\bm{1}_j
=2\lambda\bm{x}_j+\mu\bm{1}_j\,,
\qquad 1\le j\le r.
\end{equation}
We now employ matrix-notation for \eqref{03eqn:vectorequation}.
The matrix whose entries are all one is denoted by $J$ without
explicitly mentioning its size.
Similarly, the identity matrix is denoted by $I$.
Using the obvious relation
\[
\langle \bm{1}_j,\bm{x}_j\rangle\bm{1}_j=J\bm{x}_j\,,
\]
\eqref{03eqn:vectorequation} becomes
\[
(2a_j J+(2a_j-2\lambda)I)\bm{x}_j=\mu \bm{1}_j\,,
\]
or equivalently,
\begin{equation}\label{02eqn:key equation}
\left(J-\left(\frac{\lambda}{a_j}-1\right)I\right)\bm{x}_j
=\frac{\mu}{2a_j}\,\bm{1}_j\,,
\qquad 1\le j\le r.
\end{equation}
On the other hand,
\eqref{02eqn:stationary points of F(condition)}
is equivalent to conditions \eqref{03eqn:condition (1)}
and \eqref{03eqn:condition (2)}.
Consequently, we have
\[
\mathcal{S}
=\{(x_0,\bm{x},\lambda,\mu)\,\,
\text{satisfying
\eqref{03eqn:condition (1)},
\eqref{03eqn:condition (2)},
\eqref{vectorequation0} and \eqref{02eqn:key equation}}\}.
\]

\begin{lemma}\label{02lem:values at stationary point}
If $(x_0,\bm{x},\lambda,\mu)\in\mathcal{S}$,
then $\phi(x_0,\bm{x})=\lambda$.
In particular,
\begin{equation}\label{02eqn:M(b,a)01}
M(\bm{d},\bm{a})=\min\{\lambda\,;\,
(x_0,\bm{x},\lambda,\mu)\in\mathcal{S}\}.
\end{equation}
\end{lemma}

\begin{proof}
Taking the inner product of \eqref{03eqn:vectorequation}
with $\bm{x}_{j}$, we get
\[
2a_{j}\langle \bm{x}_j,\bm{x}_j\rangle
+2a_{j}\langle \bm{1}_j,\bm{x}_j\rangle^2
=2\lambda \langle \bm{x}_j,\bm{x}_j \rangle
+\mu\langle \bm{1}_j,\bm{x}_j\rangle.
\]
Taking the sum over $j$
and applying conditions \eqref{03eqn:condition (1)}
and \eqref{03eqn:condition (2)}, we obtain
\[
2\phi(x_0,\bm{x})=2\lambda(1-x_0^2)+\mu (-x_0)
=2\lambda-2\lambda x_0^2-\mu x_0
\]
and hence $\phi(x_0,\bm{x})=\lambda$ by \eqref{vectorequation0}.
Then \eqref{02eqn:M(b,a)01} is immediate from \eqref{02eqn:M(b,a)00}.
\end{proof}

Upon solving the linear equation \eqref{02eqn:key equation}
the following elementary result is useful.

\begin{lemma}\label{02lem:05}
Let $m\ge1$.
Let $J$ denote the $m\times m$ matrix whose entries are all one,
and $I$ the $m\times m$ identity matrix.
For $\alpha,\beta\in\mathbb{R}$ we consider the linear equation:
\[
(J-\alpha I)\bm{x}=\beta\bm{1}.
\]
\begin{enumerate}
\item[\upshape (i)] If $\alpha=0$, then
the solution is given by
\[
\bm{x}=\dfrac{\beta}{m}\,\1+\bm{y},
\qquad \bm{y}\in\mathrm{Ker\,} J.
\]
Moreover, $\dim \mathrm{Ker\,} J=m-1$.
In particular, the solution is unique when $m=1$.
\item[\upshape (ii)] If $\alpha=m$ and $\beta=0$,
the solution is given by $\bm{x}=c\bm{1}$ with $c\in\mathbb{R}$.
In this case, $m$ is an eigenvalue of $J$ and $\bm{x}$ is an
associated eigenvector.
\item[\upshape (iii)] If $\alpha=m$ and $\beta\neq0$, there is no solution.
\item[\upshape (iv)] If $\alpha\neq 0$ and $\alpha\neq m$, 
the solution is unique and given by
\[
\bm{x}=\frac{\beta}{m-\alpha}\,\bm{1}.
\]
\end{enumerate}
\end{lemma}

Suppose that a real number $\lambda$ appears in $\mathcal{S}$,
i.e., $(x_0,\bm{x},\lambda,\mu)\in\mathcal{S}$ for some
$x_0,\bm{x},\mu$,
and that $\lambda\not\in \{a_j\,,a_jd_j+a_j\,;\,1\le j\le r\}$.
It then follows from Lemma \ref{02lem:05} (iv) that
\eqref{02eqn:key equation} admits a unique solution
\begin{equation}\label{02eqn:sol x_j}
\bm{x}_j=\frac{\mu}{2(a_jd_j+a_j-\lambda)}\,\bm{1}_j,
\qquad
1\le j\le r.
\end{equation}
Since $\lambda\neq 0$, which is directly verified
or by Proposition \ref{03prop:M(d,a)le min},
\eqref{vectorequation0} becomes
\begin{equation}\label{02eqn:sol x_0}
x_0=-\frac{\mu}{2\lambda}\,.
\end{equation}
Inserting \eqref{02eqn:sol x_j} and \eqref{02eqn:sol x_0}
into condition \eqref{03eqn:condition (2)}, we have
\begin{equation}\label{03eqn:key equation for lambda (-1)}
\sum_{j=1}^r \frac{d_j\mu}{2(a_jd_j+a_j-\lambda)}
-\frac{\mu}{2\lambda}=0.
\end{equation}
We see from \eqref{02eqn:sol x_j} and \eqref{02eqn:sol x_0}
together with \eqref{03eqn:condition (1)} that $\mu\neq0$.
Hence \eqref{03eqn:key equation for lambda (-1)} is equivalent to
\begin{equation}\label{02eqn:key equation for lambda}
\sum_{j=1}^r \frac{d_j}{a_jd_j+a_j-\lambda}
=\frac{1}{\lambda}\,.
\end{equation}
Thus, $\lambda$ is a solution of \eqref{02eqn:key equation for lambda}.

Conversely, with any solution $\lambda$ of
\eqref{02eqn:key equation for lambda}
we may associate $\mu$ in such a way that
\eqref{02eqn:sol x_j} and \eqref{02eqn:sol x_0} satisfy
condition \eqref{03eqn:condition (1)}.
In other words, every solution $\lambda$ of
\eqref{02eqn:key equation for lambda} appears in $\mathcal{S}$.
Consequently,
\begin{align}
\{\lambda_1,\dots,\lambda_s\}
&\subset \{\lambda\,;\, (x_0,\bm{x},\lambda,\mu)\in\mathcal{S}\}
\nonumber \\
&\subset
\{\lambda_1,\dots,\lambda_s\}
\cup
\{a_j\,,\, a_jd_j+a_j\,;\, 1\le j\le r\},
\label{02eqn:determining S}
\end{align}
where $\lambda_1<\dotsb<\lambda_s$ are the solutions of
\eqref{02eqn:key equation for lambda},
see Proposition \ref{01prop:01}.

We are now in a position to determine
\[
M(\bm{d},\bm{a})
=\min\{\lambda\,;\,
(x_0,\bm{x},\lambda,\mu)\in\mathcal{S}\},
\]
see Lemma \ref{02lem:values at stationary point}.
Since
\[
M(\bm{d},\bm{a})
\le \min\{a_1,\dots,a_r\}
\]
by Proposition \ref{03prop:M(d,a)le min},
it follows from \eqref{02eqn:determining S} that
\[
\min\{\lambda\,;\,(x_0,\bm{x},\lambda,\mu)\in\mathcal{S}\}
=\min\{\lambda_1,\dots,\lambda_s\}
=\lambda_1\,,
\]
where $\lambda_1$ is the minimal solution
of \eqref{02eqn:key equation for lambda}.
Consequently, $M(\bm{d},\bm{a})=\lambda_1$ as desired.

%%%%%%%%%%%%%%%%%%%%%%%%%%%%%%%%%%%%%%%%%
\subsection{An infinite case}

\begin{proposition}\label{gotoinfinitypropmin}
Let $r\ge2$.
Assume that $a_j>0$ and $1\le d_j\le\infty$ for all $1\le j\le r$.
Then
\begin{equation}\label{03eqn:M=inf}
M(\bm{d},\bm{a})
=\inf\left\{M(\bm{e},\bm{a})\,;\,
\begin{array}{l}
\bm{e}=(e_1,\dots,e_r)\le\bm{d},\\
e_1<\infty, \dots e_r<\infty 
\end{array}
\right\} 
\end{equation}
Moreover,
\begin{equation}\label{03eqn:M=lim}
M(\bm{d},\bm{a})
=\lim_{n\to\infty}M(\bm{d}\wedge \bm{n},\bm{a}),
\end{equation}
where $\bm{d}\wedge \bm{n}
=(d_1\wedge n,\dots,d_r\wedge n)$.
\end{proposition}

\begin{proof}
Denote by $\mu$ the right-hand side of \eqref{03eqn:M=inf}.
If $\bm{e}=(e_1,\dots,e_r)$ satisfies
$e_j<\infty$ and $e_j\le d_j$ for all $1\le j\le r$,
by definition we have
$M(\bm{d},\bm{a})\le M(\bm{e},\bm{a})$.
Therefore, the inequality $M(\bm{d},\bm{a})\le\mu$ holds.
On the other hand,
for any $\epsilon>0$ 
there exists a vector
$(x_0,\bm{x})$ with finite supports
such that $\phi(x_0,\bm{x})\le M(\bm{d},\bm{a})+\epsilon$.
Choosing $\bm{e}=(e_1,\dots,e_r)$ with
$e_j<\infty$ and $e_j\le d_j$ for all $1\le j\le r$
such that
$\bm{x}\in\mathbb{R}^{e_1}
\times\dots\times\mathbb{R}^{e_r}$,
we have $M(\bm{e},\bm{a})\le \phi(x_0,\bm{x})$.
Hence $\mu\le M(\bm{e},\bm{a})\le M(\bm{d},\bm{a})+\epsilon$
so that $\mu\le M(\bm{d},\bm{a})$.
Consequently, $\mu= M(\bm{d},\bm{a})$ and \eqref{03eqn:M=inf}
is proved.
Then \eqref{03eqn:M=lim} is now immediate.
\end{proof}

We now complete the proof of 
Theorem \ref{02thm:Characterization of M(d,a)}.
Let $r\ge2$ and suppose that 
$a_j>0$ and $1\le d_j\le\infty$ for all $1\le j\le r$.
It follows from the proved part of
Theorem \ref{02thm:Characterization of M(d,a)}
that 
\[
M(\bm{d}\wedge\bm{n},\bm{a})=\lambda_1(\bm{d}\wedge\bm{n},\bm{a}).
\]
Letting $n\rightarrow\infty$ with the help of
Propositions \ref{gotoinfinitypropeq}
and \ref{gotoinfinitypropmin} we obtain
\[
M(\bm{d},\bm{a})=\lambda_1(\bm{d},\bm{a}),
\]
as desired.

%%%%%%%%%%%%%%%%%%%%%%%%%%%%%%%%%%%%%%%%%%%%%
\subsection{Estimates of $M(\bm{d},\bm{a})$}

Having established in
Theorem \ref{02thm:Characterization of M(d,a)} 
the relation $M(\bm{d},\bm{a})=\lambda_1(\bm{d},\bm{a})$,
we may apply the results in Section \ref{Sec:Preliminaries} to 
obtain various estimates of $M(\bm{d},\bm{a})$.
Here we only mention the most basic result,
which follows directly from Proposition \ref{02prop:020}.

\begin{theorem}\label{thm:key estimate}
Let $r\ge2$.
Assume that $a_j>0$ and $1\le d_j\le\infty$ for all $1\le j\le r$.
Then we have
\begin{equation}\label{eqn:key estimate}
\left(\dfrac{1}{a_1}+\dots+\dfrac{1}{a_r}\right)^{-1}
\le M(\bm{d},\bm{a})
< \min\{a_1,\dots,a_r\},
\end{equation}
where the equality holds if and only if $d_1=\dots=d_r=\infty$.
\end{theorem}

%%%%%%%%%%%%%%%%%%%%%%%%%%%%%%%%%%%%%%%%%
\section{Star product graphs}
\label{Sec:Star product graphs}

Let $r\ge1$ be a natural number.
For each $1\le j\le r$ let
$G_j = (V_j, E_j)$ be a connected graph
with distinguished vertex $o_j\in V_j$.
The star product
\begin{equation}\label{04:star product notation}
(G_1,o_1)\star\dotsb\star (G_r,o_r)
\end{equation}
is by definition a graph $G=(V,E)$ obtained by glueing
graphs $G_j$ at the vertices $o_j$.
Although the star product depends on the choice of the
distinguished vertices,
we write
\[
G=G_1\star\dotsb\star G_r
\]
whenever there is no danger of confusion.
It is convenient to understand
the set $V$ of vertices of $G=G_1\star\dotsb\star G_r$ as
a disjoint union:
\[
V=\{o\}\cup \bigcup_{j=1}^r V_j\backslash\{o_j\},
\]
where $o$ is identified with the glued vertices $o_j\in V_j$.
Let $D_j=[d_j(x,y)]$ and $D=[d(x,y)]$ be
the distance matrices of $G_j$ and $G$, respectively.
Apparently,
\begin{equation}\label{03eqn:distance of star product}
d(x,y)=\begin{cases}
d_j(x,y), & \text{if $x,y\in V_j$}, \\
d_i(x,o)+d_j(o,y),
& \text{if $x\in V_i$ and $y\in V_j$, $i\neq j$}.
\end{cases}
\end{equation}
We are interested in a good estimate
of $\mathrm{QEC}(G_1\star \dotsb\star G_r)$
in terms of $Q_j=\mathrm{QEC}(G_j)$.

We need a general notion.
Let $G=(V,E)$ be a connected graph and $H=(W,F)$ a connected subgraph.
Let $D$ and $D_H$ be the distance matrices of $G$ and $H$,
respectively.
We say that $H$ is \textit{isometrically embedded} in $G$ if
$D_H(x,y)=D(x,y)$ for any $x.y\in W$.
In that case, $H$ is the induced subgraph of $G$ spanned by $W$,
but the converse assertion is not true.

\begin{proposition}\label{04prop:isometrically embedded subgraphs}
Let $G$ be a connected graph and $H$ a connected subgraph.
If $H$ is isometrically embedded in $G$,
we have $\mathrm{QEC}(H)\le\mathrm{QEC}(G)$.
\end{proposition}

\begin{proof}
Straightforward from definition,
see also \cite{Obata-Zakiyyah2017}.
\end{proof}

\begin{proposition}\label{04prop:estimate from below}
Let $r\ge1$.
For $1\le j\le r$ let $G_j = (V_j, E_j)$ be a 
(finite or infinite) connected graph.
Then we have
\[
\max\{Q_1,\dots,Q_r\}\le\mathrm{QEC}(G_1\star\dots\star G_r).
\]
\end{proposition}

\begin{proof}
It is obvious by definition of
star product each $G_j$ is isometrically embedded in 
$G=G_1\star\dotsb\star G_r$,
see also \eqref{03eqn:distance of star product}.
Then by Proposition
\ref{04prop:isometrically embedded subgraphs},
we have $Q_j\le \mathrm{QEC}(G)$ for all $1\le j\le r$
and hence $\max\{Q_1,\dots,Q_r\}\le \mathrm{QEC}(G)$.
\end{proof}

An estimate $\mathrm{QEC}(G_1\star\dots\star G_r)$
from above is much harder to obtain.
We start with the case where all factors $G_j$ are finite graphs.

\begin{proposition}\label{03prop:general star products}
Let $r\ge1$.
For $1\le j\le r$ let $G_j = (V_j, E_j)$ be a 
connected graph on $n_j+1=|V_j|\ge2$ vertices
($n_j=\infty$ may happen) with QE constant $Q_j=\mathrm{QEC}(G_j)$.
Let $M=M(n_1,\dots,n_r;-Q_1,\dots,-Q_r)$ be the conditional infimum of
\begin{equation}\label{04eqn:quadratic function}
\phi(x_0,\bm{x})
=\sum_{j=1}^r (-Q_j)\,
\Big\{\langle \bm{x}_j, \bm{x_j}\rangle
  +\langle \bm{1},\bm{x}_j\rangle^2 \Big\},
\qquad
x_0\in\mathbb{R},
\quad
\bm{x}_j\in\mathbb{R}^{n_j},
\end{equation}
subject to
\begin{align}
&x_0^2+\sum_{j=1}^r \langle \bm{x}_j,\bm{x}_j\rangle=1,
\label{03eqn:condition 2}\\
&x_0+\sum_{j=1}^r \langle \bm{1}_j,\bm{x}_j\rangle=0.
\label{03eqn:condition 1}
\end{align}
Then we have
\begin{equation}\label{03eqn:general estimate}
\mathrm{QEC}(G_1\star\dotsb\star G_r)
\le -M.
\end{equation}
\end{proposition}

\begin{proof}
Set $G=G_1\star\dotsb \star G_r$ and $Q=\mathrm{QEC}(G)$ for simplicity.
We keep the notations
introduced in the first paragraph of this section.
Given $f\in C_0(V)$ satisfying
\begin{equation}\label{04eqn:cond (01)}
\langle f,f\rangle=1,
\qquad
\langle\bm{1},f\rangle=0,
\end{equation}
we define $f_j\in C_0(V)$ by
\begin{equation}\label{03eqn:def f_j}
f_j(x)=
\begin{cases}
f(x), & x\in V_j\backslash \{o_j\}, \\
-\displaystyle\sum_{x\in V_j\backslash\{o_j\}}^{\mathstrut}
f(x), & x=o, \\
0, &\text{otherwise}.
\end{cases}
\end{equation}
Using $\langle \1, f\rangle=0$ we obtain easily
\begin{equation}\label{03eqn:sum of f_j}
f(x)=\sum_{j=1}^r f_j(x),
\qquad x\in V.
\end{equation}
We show that
\begin{equation}\label{2eqn:splitting quadratic function D}
\langle f,Df\rangle
=\sum_{j=1}^r \langle f_j,D_jf_j\rangle_{V_j}\,.
\end{equation}
In fact, using \eqref{03eqn:sum of f_j} we have
\begin{equation}\label{03eqn:splitting quadratic function D(0)}
\langle f,Df\rangle
=\sum_{i,j=1}^r \langle f_i,Df_j\rangle
=\sum_{j=1}^r \langle f_j,Df_j\rangle
  +\sum_{i\neq j} \langle f_i,Df_j\rangle
\end{equation}
Since $f_j$ vanishes outside $V_j$, we have
\begin{align}
\langle f_j,Df_j\rangle
&=\sum_{x,y\in V_j} d(x,y)f_j(x)f_j(y)
\nonumber \\
&=\sum_{x,y\in V_j} d_j(x,y)f_j(x)f_j(y)
=\langle f_j,D_jf_j\rangle_{V_j}\,.
\label{03eqn:in proof (00)}
\end{align}
On the other hand, for $i\neq j$ using
\eqref{03eqn:distance of star product} and
\eqref{03eqn:condition for f_j} we obtain
\begin{align}
\langle f_i,Df_j\rangle
&=\sum_{x,y\in V} d(x,y)f_i(x)f_j(y)
\nonumber \\
&=\sum_{x\in V_i}\sum_{y\in V_j} (d_i(x,o)+d_j(o,y))f_i(x)f_j(y)
\nonumber \\
&=\sum_{x\in V_i}d_i(x,o)f_i(x)\sum_{y\in V_j}f_j(y)
 +\sum_{x\in V_i}f_i(x) \sum_{y\in V_j} d_j(o,y)f_j(y)
\nonumber \\
&=\langle \1_j, f_j\rangle_{V_j} \sum_{x\in V_i}d_i(x,o)f_i(x)
 +\langle \1_i, f_i\rangle_{V_i} \sum_{y\in V_j} d_j(o,y))f_j(y)
\nonumber \\
&=0.
\label{03eqn:in proof (02)}
\end{align}
Inserting \eqref{03eqn:in proof (00)} and
\eqref{03eqn:in proof (02)} into
\eqref{03eqn:splitting quadratic function D(0)},
we obtain \eqref{2eqn:splitting quadratic function D}.

Each $f_j$ defined by \eqref{03eqn:def f_j}
being regarded as a function in $C_0(V_j)$, we have
\begin{equation}\label{03eqn:condition for f_j}
\langle \1_j,f_j\rangle_{V_j}
=\sum_{x\in V_j}f_j(x)
=0.
\end{equation}
Then we have
\[
\langle f_j,D_jf_j\rangle_{V_j}
\le Q_j \langle f_j, f_j\rangle_{V_j}.
\]
and by \eqref{2eqn:splitting quadratic function D},
\begin{equation}\label{03eqn:in proof (03)}
\langle f,Df\rangle
\le \sum_{j=1}^r Q_j \langle f_j, f_j\rangle_{V_j}.
\end{equation}
Employing vector-notation,
we associate $(x_0,\bm{x}_1,\dots,\bm{x}_r)$ with
each $f\in C_0(V)$ in such a way that
\[
x_0=f(o),
\qquad
\bm{x}_j=\big[f(x)\,;\, x\in V_j\backslash\{o\} \big]
\in\mathbb{R}^{n_j}.
\]
Then every $\bm{x}_j$ has a finite support, and we come to
\begin{align*}
\langle f_j, f_j\rangle_{V_j}
&=\langle \bm{x}_j, \bm{x_j}\rangle +f_j(o)^2 \\
&=\langle \bm{x}_j, \bm{x_j}\rangle
+\Bigg(-\sum_{x\in V_j\backslash\{o_j\}}f(x)\Bigg)^2 \\
&=\langle \bm{x}_j, \bm{x_j}\rangle
  +\langle \bm{1}_j,\bm{x}_j\rangle^2.
\end{align*}
Then \eqref{03eqn:in proof (03)} becomes
\begin{equation}\label{3eqn:estimate of QEC}
\langle f,Df\rangle
\le \sum_{j=1}^r Q_j
\Big\{\langle \bm{x}_j, \bm{x_j}\rangle
  +\langle \bm{1},\bm{x}_j\rangle^2 \Big\},
\end{equation}
or equivalently,
\begin{equation}\label{04eqn:in proof (101)}
\langle f,Df\rangle
\le -\phi(x_0,\bm{x}),
\end{equation}
for any $f\in C_0(V)$ satisfying \eqref{04eqn:cond (01)},
which is equivalent to
\eqref{03eqn:condition 2} and \eqref{03eqn:condition 1}.
By definition of the QE constant, for any $\epsilon>0$
there exists $f\in C_0(V)$ satisfying \eqref{04eqn:cond (01)} such that
\[
Q-\epsilon\le \langle f,Df\rangle.
\]
In view of \eqref{04eqn:in proof (101)} we obtain
\[
Q-\epsilon\le -\phi(x_0,\bm{x})\le -M,
\]
where we used the obvious inequality 
$\phi(x_0,\bm{x})\ge M$ for any $(x_0,\bm{x})$ satisfying 
\eqref{03eqn:condition 2} and \eqref{03eqn:condition 1}.
Consequently, $Q\le -M$ as desired.
\end{proof}

We are now in a position to state the main results.

\begin{theorem}\label{01thm:QEC=0}
Let $r\ge1$.
For $1\le j\le r$ let $G_j = (V_j, E_j)$ be a 
connected graph on $|V_j|\ge2$ vertices ($|V_j|=\infty$ may happen).
Assume that every $G_j$ is of QE class with QE constant
$Q_j=\mathrm{QEC}(G_j)\le0$.
If $Q_j=0$ for some $j$, we have
$\mathrm{QEC}(G_1\star\dots\star G_r)=0$.
\end{theorem}

\begin{proof}
We apply Proposition \ref{03prop:general star products}.
By assumption the coefficients $-Q_j$ 
in the right-hand side of \eqref{04eqn:quadratic function} 
are all non-negative, and at least one $-Q_j$ vanishes.
It then follows from Proposition \ref{03prop:trivial case} that
the conditional infimum is zero, that is, $M=0$.
Hence by \eqref{03eqn:general estimate} we have
$\mathrm{QEC}(G_1\star\dotsb \star G_r)\le 0$.
On the other hand, it follows from
Proposition \ref{04prop:estimate from below} that
\[
0=\max\{Q_1,\dots,Q_r\}
\le \mathrm{QEC}(G_1\star\dotsb \star G_r).
\]
Hence $\mathrm{QEC}(G_1\star\dotsb \star G_r)=0$.
\end{proof}

\begin{theorem}\label{04thm:main estimate}
Let $r\ge1$.
For $1\le j\le r$ let $G_j = (V_j, E_j)$ be a 
connected graph on $n_j+1=|V_j|\ge2$ vertices ($n_j=\infty$ may happen).
Assume that every $G_j$ is of QE class with QE constant
$Q_j=\mathrm{QEC}(G_j)<0$.
Then we have
\begin{equation}\label{01eqn:main estimate}
\max\{Q_1,\dots,Q_r\}
\le\mathrm{QEC}(G_1\star\dots\star G_r)
\le -\Lambda,
\end{equation}
where $\Lambda$ is the minimal solution of
\begin{equation}\label{01eqn:Lambda as minimal solution}
\sum_{j=1}^r \frac{n_j}{-Q_jn_j-Q_j-\lambda}=\frac{1}{\lambda}\,.
\end{equation}
\end{theorem}

\begin{proof}
The left half of \eqref{01eqn:main estimate} is already 
shown in Proposition \ref{04prop:estimate from below}.
We will show the right half.
We first see from Proposition \ref{03prop:general star products} that
\[
\mathrm{QEC}(G_1\star\dotsb\star G_r)
\le -M,
\]
where $M=M(n_1,\dots,n_r;-Q_1,\dots,-Q_r)$
is the conditional infimum of \eqref{04eqn:quadratic function}
subject to \eqref{03eqn:condition 2} and \eqref{03eqn:condition 1}.
On the other hand, in case where $Q_j<0$ for all $1\le j\le r$,
$M$ coincides with the minimal solution of
\eqref{01eqn:Lambda as minimal solution} by Theorem
\ref{02thm:Characterization of M(d,a)}.
Thus, \eqref{01eqn:main estimate} follows.
\end{proof}

\begin{corollary}\label{04cor:r-star product}
We keep the notations and assumptions as in Theorem \ref{04thm:main estimate}.
If $r\ge2$, we have
\begin{equation}
\mathrm{QEC}(G_1\star\dots\star G_r)
\le\left(\frac{1}{Q_1}+\dotsb+\frac{1}{Q_r}\right)^{-1}<0.
\end{equation}
\end{corollary}

\begin{proof}
Immediate from Theorems \ref{thm:key estimate}
and \ref{04thm:main estimate}.
\end{proof}

\begin{corollary}\label{3cor:strict}
For $j=1,2$ let $G_j = (V_j, E_j)$
be a (finite or infinite) connected graph on $n_j+1=|V_j|\ge2$ vertices.
Assume that each $G_j$ is of QE class with QE constant
$Q_j=\mathrm{QEC}(G_j)<0$.
Then we have
\begin{equation}\label{04eqn:main estimate}
\max\{Q_1,Q_2\}
\le\mathrm{QEC}(G_1\star G_2)
\le Q_{12},
\end{equation}
where $Q_{12}$ is defined by
\begin{equation}\label{01eqn:Q_12}
Q_{12}=\frac{2Q_1Q_2}{Q_1+Q_2-\sqrt{(Q_1+Q_2)^2
-\dfrac{4(n_1+n_2+1)}{(n_1+1)(n_2+1)}\, Q_1Q_2}}\,.
\end{equation}
Moreover,
\begin{equation}\label{01eqn:estimate for Q_12}
\max\{Q_1,Q_2\}< Q_{12}<0.
\end{equation}
\end{corollary}

\begin{proof}
\eqref{04eqn:main estimate} is
a direct consequence of Theorem \ref{04thm:main estimate}
and \eqref{01eqn:estimate for Q_12} is verified directly.
\end{proof}

\begin{remark}\normalfont
If $n_1<n_2=\infty$, the right-hand side of \eqref{01eqn:Q_12} 
is replaced with the limit as $n_2\rightarrow\infty$.
If $n_1=n_2=\infty$, \eqref{01eqn:Q_12} is understood as
\[
Q_{12}=\frac{Q_1Q_2}{Q_1+Q_2}\,.
\]
\end{remark}

We give some examples in connection with
inequality \eqref{04eqn:main estimate}.

\begin{example}\normalfont
Let $K_3$ be the complete graph on three vertices.
The star product $K_3\star K_3$ is illustrated
in Figure \ref{Fig:star products}.
It is known that $\mathrm{QEC}(K_3)=-1$.
Inserting $Q_1=Q_2=-1$ and $n_1+1=n_2+1=3$ into \eqref{01eqn:Q_12},
we have
\[
Q_{12}=-\frac{3}{5}\,.
\]
On the other hand, by a direct verification we have
\[
\mathrm{QEC}(K_3\star K_3)=-\frac35\,,
\]
see also \cite[Sect. 5.2, No. 11]{Obata-Zakiyyah2017}.
In this case we have
\[
\max\{Q_1,Q_2\}<\mathrm{QEC}(K_3\star K_3)=Q_{12}<0.
\]
\end{example}

\begin{figure}[hbt]
\begin{center}
\includegraphics[width=64pt]{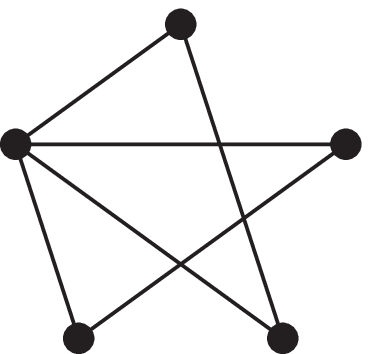}
\qquad\qquad
\includegraphics[width=64pt]{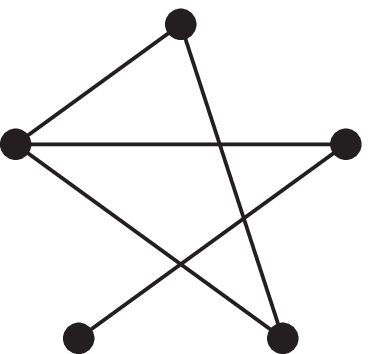}
\qquad\qquad
\includegraphics[width=64pt]{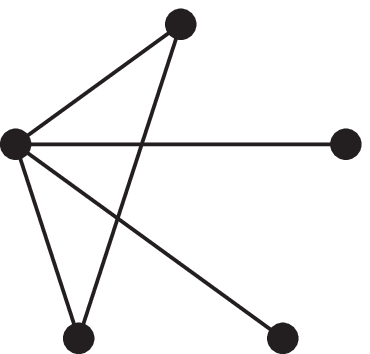}
\end{center}
\caption{$K_3\star K_3$ (left), $G_1$ (middle) and $G_2$ (right)}
\label{Fig:star products}
\end{figure}

\begin{example}\label{04ex:non-isomorphic star products}
\normalfont
We consider $K_3\star P_3$, where $P_3$ is the path on three vertices.
There are two non-isomorphic star products in this case,
say, $G_1$ and $G_2$ as shown in Figure \ref{Fig:star products}.
It is known that $\mathrm{QEC}(K_3)=-1$ and $\mathrm{QEC}(P_3)=-2/3$.
Inserting $Q_1=-1$, $Q_2=-2/3$, $n_1+1=n_2+1=3$
into \eqref{01eqn:Q_12}, we have
\[
Q_{12}=\frac{-15+\sqrt{105}}{10}=-\frac{12}{15+\sqrt{105}}
\approx -0.4753.
\]
On the other hand, it follows by a direct calculation
(see also \cite[Sect.~5.2, No.~4 and No.~7]{Obata-Zakiyyah2017})
that
\[
\mathrm{QEC}(G_1)=-\frac{6}{6+\sqrt{21}}\approx -0.5670,
\qquad
\mathrm{QEC}(G_2)=-\frac{12}{15+\sqrt{105}}\,.
\]
Thus, we obtain an interesting contrast:
\begin{align*}
&\max\{Q_1,Q_2\}<\mathrm{QEC}(G_1)<Q_{12}<0, \\
&\max\{Q_1,Q_2\}<\mathrm{QEC}(G_2)=Q_{12}<0.
\end{align*}
\end{example}

\begin{example}\normalfont
It is known that $\mathrm{QEC}(K_2)=-1$ and $\mathrm{QEC}(C_4)=0$,
where $C_4$ is the cycle on four vertices.
It follows from Theorem \ref{01thm:QEC=0}
that $\mathrm{QEC}(K_2\star C_4)=0$.
On the other had,
inserting $Q_1=-1$, $Q_2=0$,
$n_1+1=2$ and $n_2+1=4$ into \eqref{01eqn:Q_12},
we have $Q_{12}=0$.
Thus we have
\[
\max\{Q_1,Q_2\}=\mathrm{QEC}(K_2\star C_4)=Q_{12}=0.
\]
\end{example}

\begin{figure}[hbt]
\begin{center}
\includegraphics[width=64pt]{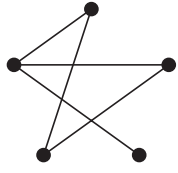}
\end{center}
\caption{$K_2\star C_4$}
\label{Fig: K_2 star C_4}
\end{figure}

Along with the above observation,
a natural question arises to determine
the extremal classes of star products
$G_1\star G_2$ such that
\[
\mathrm{QEC}(G_1\star G_2)=Q_{12}
\]
and 
\[
\mathrm{QEC}(G_1\star G_2)=\max\{Q_1,Q_2\}.
\]
Remind that the star product depends also
on the choice of distinguished vertices $o_1$ and $o_2$,
as is illustrated in Example
\ref{04ex:non-isomorphic star products}.

%%%%%%%%%%%%%%%%%%%%%%%%%%%%%%%%%%%%%%%%%%%%%%%%%%%%%%%%%%%%%%%%%%%%
\section{Infinite graphs}
\label{Sec:Infinite graphs}

\subsection{A limit formula}

\begin{proposition}\label{04prop:infinite graphs}
Let $G=(V,E)$ be a connected graph.
Let $H_n=(W_n,F_n)$ be a sequence of connected subgraphs of $G$
such that $W_1\subset W_2\subset\dotsb$ and
$V=\bigcup_{n=1}^\infty W_n$.
If each $H_n$ is isometrically embedded in $G$,
we have
\begin{equation}\label{05eqn:limit of subgraphs}
\mathrm{QEC}(G)=\lim_{n\rightarrow\infty}\mathrm{QEC}(H_n).
\end{equation}
\end{proposition}

\begin{proof}
Let $D$ denote the distance matrix of $G$.
By definition, for any $\epsilon>0$ there exists
$f\in C_0(V)$ such that $\langle f,f\rangle=1$,
$\langle \bm{1},f\rangle=0$ and
$\langle f,Df\rangle\ge \mathrm{QEC}(G)-\epsilon$.
By assumption we may choose $n_0$ such that 
$f(x)=0$ outside of $W_n$ for all $n\ge n_0$.
Then $\mathrm{QEC}(H_n)\ge \langle f,Df\rangle$ for all $n\ge n_0$
and we have
\[
\mathrm{QEC}(G)-\epsilon\le \mathrm{QEC}(H_n),
\qquad n\ge n_0.
\]
On the other hand, it follows from 
Proposition \ref{04prop:isometrically embedded subgraphs} that
\[
\mathrm{QEC}(H_n)\le \mathrm{QEC}(G)
\]
Consequently, \eqref{05eqn:limit of subgraphs} holds.
\end{proof}

\begin{proposition}
Any (finite or infinite) tree is of QE class.
\end{proposition}

\begin{proof}
For any tree we may choose a sequence of finite subtrees
of which the union covers the whole tree.
Note that any subtree of a tree is isometrically embedded.
Then, in view of Proposition \ref{04prop:infinite graphs}
it is sufficient to show that every finite tree is of QE class.
More precisely, for a finite tree $G=(V,E)$ on $|V|\ge3$ vertices we have
\begin{equation}\label{05eqn:QEC of tree}
\mathrm{QEC}(G)<-\frac{1}{|V|-1}\,.
\end{equation}
In fact, a tree on $n$ vertices is represented as
$G=G_1\star\dots\star G_{n-1}$, where each $G_j$ is isomorphic to $K_2$.
Note that $Q_j=\mathrm{QEC}(G_j)=\mathrm{QEC}(K_2)=-1$.
Then by Corollary \ref{04cor:r-star product} we obtain
\[
\mathrm{QEC}(G)
=\mathrm{QEC}(G_1\star \dots\star G_{n-1})
<\left(\frac{1}{Q_1}+\dots+\frac{1}{Q_{n-1}}\right)^{-1}
=-\frac{1}{n-1}\,,
\]
as desired.
\end{proof}

The above result is a reproduction of Haagerup \cite{Haagerup79}.
The estimate \eqref{05eqn:QEC of tree} is far from best possible.
It is an interesting question to determine the QE constant of a tree.

\begin{proposition}
Let $K_\infty$ be the infinite complete graph,
that is, the graph on a countably infinite set such that any pair
of distinct vertices are connected by an edge.
Then $\mathrm{QEC}(K_\infty)=-1$.
\end{proposition}

\begin{proof}
Every finite subgraph of $K_{\infty}$
is of the form $K_{n}$ and $\mathrm{QEC}(K_n)=-1$.
Now we apply Proposition~\ref{04prop:infinite graphs}.
\end{proof}

\subsection{The path graphs $P_n$}

For $n\ge1$ let $P_n$ be the path graph on
the vertex set $V=\{0,1,2,\dots,n-1\}$
and edge set $E=\{\{0,1\},\{1,2\},\dots,\{n-2,n-1\}\}$.
Let $D=[d(i,j)]$ be the distance matrix as usual.
Note that $d(i,j)=|i-j|$ for $i,j\in V$.
We start with the following

\begin{proposition}\label{propositionpnmatrix}
For $n\ge1$ let $c_n$ be the maximal number $c$ such that 
the $n\times n$ matrix
\begin{equation}\label{matrixforpnplus1}
\left[2\min\{i,j\}-c-c\cdot\delta_{ij}\right]_{i,j=1}^n
\end{equation}
is positive definite.
Then $\mathrm{QEC}(P_{n+1})=-c_n$.
\end{proposition}

\begin{proof}
Suppose $f\in C(V)$ satisfies $\langle\bm{1},f\rangle=0$.
Then we have
\begin{align}
\langle f,Df\rangle
&=\sum_{i,j=0}^{n} |i-j|f(i)f(j) 
\nonumber\\
&=\sum_{i=1}^{n} if(i)f(0)+\sum_{j=1}^{n} jf(0)f(j)
 +\sum_{i,j=1}^n |i-j|f(i)f(j).
\label{05eqn:in proof 5.4(1)}
\end{align}
For $1\le i\le n$ we set $x_i=f(i)$.
Since $f(0)=-x_1-\dots-x_n$, \eqref{05eqn:in proof 5.4(1)} becomes
\begin{equation}\label{05eqn:in proof 5.4(2)}
\langle f,Df\rangle
=\sum_{i,j=1}^{n}(-i-j+|i-j|)x_ix_j
=-\sum_{i,j=1}^{n}2\min\{i,j\}\,x_ix_j
\end{equation}
On the other hand, we have
\begin{equation}\label{05eqn:in proof 5.4(3)}
\langle f,f\rangle
=\sum_{i=0}^{n} f(i)^2
=\left(\sum_{i=1}^{n}x_i\right)^2+\sum_{i=1}^{n}x_i^2
=\sum_{i,j=1}^{n}(1+\delta_{ij})x_i x_j.
\end{equation}
The QE constant is the minimal constant $Q\in\mathbb{R}$ 
such that $\langle f,Df\rangle\le Q\langle f,f\rangle$
for all $f\in C(V)$ with $\langle \bm{1},f\rangle=0$,
or using \eqref{05eqn:in proof 5.4(2)} and
\eqref{05eqn:in proof 5.4(3)},
\[
-\sum_{i,j=1}^{n} 2\min\{i,j\} x_i x_j
\le Q\sum_{i,j=1}^{n}(1+\delta_{ij})x_i x_j
\]
holds for every choice of $x_1,\dots,x_n\in\mathbb{R}$,
In other words, $Q$ coincides with $-c$, where
$c\in\mathbb{R}$ is the maximal constant such that
\[
\sum_{i,j=1}^{n} (2\min\{i,j\}-c(1+\delta_{ij}))x_i x_j\ge0
\]
for every choice of $x_1,\dots,x_n\in\mathbb{R}$.
This completes the proof.
\end{proof}

By direct application of Proposition \ref{propositionpnmatrix}
we obtain
\begin{align*}
-\mathrm{QEC}(P_2)&=1,\\
-\mathrm{QEC}(P_3)&=2/3,\\
-\mathrm{QEC}(P_4)&=2-\sqrt{2}=0.585786\ldots,\\
-\mathrm{QEC}(P_5)&=(5-\sqrt{5})/5=0.552786\ldots,\\
-\mathrm{QEC}(P_6)&=4-2\sqrt{3}=0.535898\ldots,\\
-\mathrm{QEC}(P_7)&=0.526048\ldots,\\
-\mathrm{QEC}(P_8)&=4+2\sqrt{2}-\sqrt{20+14\sqrt{2}}=0.519783\ldots,\\
-\mathrm{QEC}(P_9)&=0.515546\ldots,\\
-\mathrm{QEC}(P_{10})
&=6+2\sqrt{5}-\sqrt{50+22\sqrt{5}}=0.512543\ldots.
\end{align*}
The numbers $-\mathrm{QEC}(P_7)$
and $-\mathrm{QEC}(P_9)$ are the smallest real roots
of the cubic equations
\[
7c^3-28c^2+28c-8=0,
\qquad
3c^3-18c^2+24c-8=0,
\]
respectively.

Now define a family of matrices:
$A_n=\left[4\min\{i,j\}-1-\delta_{ij}\right]_{i,j=1}^{n}$, where
$1\le n\le\infty$. In particular
\[
A_{\infty}=
\begin{bmatrix}
2&3&3&3&3&\cdots\\
3&6&7&7&7&\cdots\\
3&7&10&11&11&\cdots\\
3&7&11&14&15&\cdots\\
3&7&11&15&18&\cdots\\
\vdots&\vdots&\vdots&\vdots&\vdots&\ddots
\end{bmatrix}.
\]

\begin{proposition}
For $n\ge1$,
\[
\det A_{n}=n+1.
\]
Consequently,
$A_{\infty}$ is positive definite
as well as $A_n$ for all $n\ge1$.
\end{proposition}

\begin{proof}
We are going to prove a slightly more general statement.
For $n\ge1$ and $u\in\mathbb{R}$ we define an auxiliary matrix
$A_n(u)=[u_{ij}]_{i,j=1}^{n}$, 
where
\[
u_{ij}=
\begin{cases}
4\min\{i,j\}-1-\delta_{ij}\,, & (i,j)\neq(n,n), \\
u, & (i,j)=(n,n).
\end{cases}
\]
Then $A_{n}=A_{n}(4n-2)$.
We will prove that
\begin{equation}\label{auxiliarydeterminant}
\det A_{n}(u)=nu-(n-1)(4n+1).
\end{equation} 
This is true for $n=1$.
Assume that (\ref{auxiliarydeterminant}) holds for $n-1$.
Let $\bm{k}_j$ denote the $j$th column of $A_{n}(u)$. Then
\[
\det A_n(u)
=\det[\bm{k}_1\,, \dots, \bm{k}_n]
=\det[\bm{k}_1\,, \dots, \bm{k}_{n-1}\,, \bm{k}_{n}-\bm{k}_{n-1}].
\]
Now we observe that
\[
\bm{k}_{n}-\bm{k}_{n-1}
=[0, \dots, 0, 1, u-4n+5]^{\mathrm{T}},
\]
so expanding the determinant over the last column and applying the inductive assumption we get
\begin{align*}
\det A_n(u)&=(u-4n+5)\det A_{n-1}-\det A_{n-1}(4n-5)\\
&=(u-4n+5)n-(n-1)(4n-5)+(n-2)(4n-3)\\
&=nu-(n-1)(4n+1),
\end{align*}
hence \eqref{auxiliarydeterminant} holds for $n$.
\end{proof}

\begin{theorem}\label{06thm:estimate for QEC(P_n)}
For $n\ge2$ we have
\begin{equation}\label{estimationqecpn}
-\frac{2n^4+20n^2-7+15(-1)^n}{4n^4-4+15n+15n(-1)^n}
\le \mathrm{QEC}(P_n)
\le-\frac{1}{2}.
\end{equation}
\end{theorem}

\begin{proof}
For the right half of \eqref{estimationqecpn}
it suffices to note that for $c=1/2$
the matrix $A_n$ is a multiple by $2$ of the matrix given by
\eqref{matrixforpnplus1}.

We will prove the left half of \eqref{estimationqecpn}.
Suppose that the matrix
\[
\left[2\min\{i,j\}-c-c\cdot\delta_{ij}\right]_{i,j=1}^{n-1}
\]
is positive definite.
Then for $a_i^{(n)}=i(n-i)(-1)^i$ we have
\begin{equation}\label{05eqn:in proof 5.6(1)}
\sum_{i,j=1}^{n-1}
\left(2\min\{i,j\}-c-c\cdot\delta_{ij}\right)a_{i}^{(n)}a_{j}^{(n)}
\ge0.
\end{equation}
The above sum is calculated with the help of
Lemma \ref{auxiliarylemma} in the Appendix as follows:
\begin{align*}
&\sum_{i,j=1}^{n-1}
\left(2\min\{i,j\}-c-c\cdot\delta_{ij}\right)a_{i}^{(n)}a_{j}^{(n)} \\
&=2\sum_{i,j=1}^{n-1}\min\{i,j\}\, i(n-i)j(n-j)(-1)^{i+j} \\
&\qquad -c\sum_{i,j=1}^{n-1} i(n-i)j(n-j)(-1)^{i+j}
-c\sum_{i=1}^{n-1}i^2(n-i)^2 \\
&=\frac{n}{120}\left\{2n^4+20n^2-7+15(-1)^n\right\}
  -\frac{c}{8}\left\{1+(-1)^n\right\}n^2-c\,\frac{n^5-n}{30}.
\end{align*}
Then, \eqref{05eqn:in proof 5.6(1)} yields
\[
c\le\frac{2n^4+20n^2-7+15(-1)^n}{4n^4-4+15n+15n(-1)^n}\,,
\]
which, in view of Proposition~\ref{propositionpnmatrix},
proves \eqref{estimationqecpn}.
\end{proof}

Let $\mathbb{Z}$ be the one-dimensional integer lattice,
i.e., the two-sided infinite path on the integers,
and $\mathbb{Z}_+$ be the one-sided infinite path on
$\{0,1,2,\dots\}$.

\begin{theorem}
$\mathrm{QEC}(\mathbb{Z}_+)
=\mathrm{QEC}(\mathbb{Z})=-\dfrac12$\,.
\end{theorem}

\begin{proof}
Every finite connected subgraph of $\mathbb{Z}_{+}$ and $\mathbb{Z}$
is of the form $P_{n}$ and $n$ can be arbitrarily large.
Therefore our statement is a consequence of
Theorem~\ref{06thm:estimate for QEC(P_n)} and
Proposition~\ref{04prop:infinite graphs}.
\end{proof}

\section{Appendix}
\label{Sec:Appendix}

\subsection{Some combinatorial identities}

In this part we are going to prove three identities which were used in the
proof of Theorem~\ref{06thm:estimate for QEC(P_n)}.

\begin{lemma}\label{auxiliarylemma}
For $n\ge1$ we have
\begin{align}
&\sum_{i,j=1}^{n}\min\{i,j\}\, i(n-j)j(n-j)(-1)^{i+j}
=\frac{n}{240}\left\{2n^4+20n^2-7+15(-1)^n\right\},
\label{formulaone}\\
&\sum_{i,j=1}^{n} i(n-i)j(n-j)(-1)^{i+j}
=\frac{1}{8}\left(1+(-1)^n\right)n^2,
\label{formulatwo}\\
&\sum_{i=1}^{n}i^2(n-i)^2
=\frac{n^5-n}{30}\,.
\label{formulathree}
\end{align}
\end{lemma}

\begin{proof}
For $n=0$ the identities remain true
understanding that the left-hand sides are zero.
The above three identities are used
in the proof of Theorem \ref{06thm:estimate for QEC(P_n)}.
For the proofs we will apply well-known formulas for the sums:
\begin{align*}
\sum_{i=1}^{n}i&=\frac{1}{2}n(n+1),&
\sum_{i=1}^{n}i^2&=\frac{1}{6}n(n+1)(2n+1),\\
\sum_{i=1}^{n}i^3&=\frac{1}{4}n^2(n+1)^2,&
\sum_{i=1}^{n}i^4&=\frac{1}{30}(n+1)(2n+1)(3n^2+3n-1),
\end{align*}
and also the following elementary identities:
\begin{align*}
\sum_{i=1}^{n} i(-1)^i
&=\frac{1}{4}\left(2n(-1)^n+(-1)^{n}-1\right),\\
\sum_{i=1}^{n} i^2(-1)^i
&=\frac{1}{2}n(n+1)(-1)^{n},\\
\sum_{i=1}^{n} i^3(-1)^i
&=\frac{1}{8}\left\{4n^3(-1)^n+6n^2(-1)^{n}-(-1)^j+1\right\}.
\end{align*}
Now we prove \eqref{formulaone}.
Put
\begin{align*}
A_{j}&=\sum_{i=1}^{j} i^2(n-i)j(n-j)(-1)^{i+j},\\
B_j&=\sum_{i=j+1}^{n} i(n-i)j^2(n-j)(-1)^{i+j}.
\end{align*}
By elementary calculations we find that
\begin{align*}
A_j&=\frac{j(n-j)}{8}\left\{4j^2n+4jn-4j^3-6j^2+1-(-1)^{j}\right\},\\
B_j&=\frac{j^2(n-j)}{4}\left\{2j^2+2j-n-2jn-(-1)^{j+n}n\right\}.
\end{align*}
Then we have
\begin{align*}
A_j+B_j
&=\sum_{i=1}^{n}\min\{i,j\}\,i(n-i)j(n-j)(-1)^{i+j} \\
&=\frac{1}{8}j(n-j)\left\{2jn-2j^2+1-2(-1)^{j+n}jn-(-1)^j\right\}.
\end{align*}
Summing up both sides over $j=1,2,\dots, n$,
we get \eqref{formulaone}.

Relation \eqref{formulatwo} follows from
\begin{align*}
\sum_{i,j=1}^{n} i(n-i)j(n-j)(-1)^{i+j}
&=\left(\sum_{i=1}^{n} i(n-i)(-1)^{i}\right)^2 \\
&=\left(\frac{-\left(1+(-1)^n\right)n}{4}\right)^2
=\frac{\left(1+(-1)^n\right)n^2}{8}.
\end{align*}
Relation \eqref{formulathree} can be shown in a similar manner.
\end{proof}

\subsection{A new integer sequence}

For $n\ge0$ let $a_n$ be the number given by \eqref{formulaone},
i.e.,
\begin{align}
a_n
&=\sum_{i,j=1}^{n}\min\{i,j\}\, i(n-j)j(n-j)(-1)^{i+j}
\nonumber \\
&=\frac{n}{240}\left\{2n^4+20n^2-7+15(-1)^n\right\}.
\label{06eqn:neq integer sequence}
\end{align}
Then the sequence $\{a_n\}_{n=0}^{\infty}$ begins with
\[
0, 0, 1, 4, 14, 36, 83, 168, 316, 552, 917, 1452,
2218, 3276, 4711, 6608,\dots
\]
and is absent in OEIS \cite{OEIS}.
Applying formula
\[
\sum_{n=1}^{\infty}n^N z^n=\frac{z P_{N}(z)}{(1-z)^{N+1}},
\]
where $P_{N}(z)$ are the classical Eulerian polynomials,
we can compute the generating function:
\begin{equation}\label{06eqn:generating function of new integer sequence}
\sum_{n=0}^{\infty} a_n z^n=\frac{z^2(1+z^2)^2}{(1+z)^2(1-z)^6}.
\end{equation}

Denote the ceiling of $n^2/2$ by $b_n=\lceil n^2/2\rceil$.
This sequence appears in OEIS as A000982. Now we observe that $a_n$
is the convolution of the sequence $b_n$ with itself.

\begin{proposition}
For every $n\ge0$ we have $a_n=\sum_{k=0}^{n} b_{k}b_{n-k}$.
\end{proposition}

\begin{proof}
The generating function for $a_n$
is the square of
\[
\frac{z(1+z^2)}{(1+z)(1-z)^3},
\]
which is the generating function for $b_n$, see entry A000982 in OEIS.
\end{proof}

%%%%%%%%%%%%%%%%%%%%%%%%%%%%%%%%%%%%%%%%%%%

\end{document}